\documentclass{amsart}
\usepackage[utf8]{inputenc}
\usepackage{amsfonts}
\usepackage{amsmath}
\usepackage{mathtools}
\usepackage{amsthm, mathrsfs}
\usepackage{amssymb,color}
\usepackage{latexsym}
\usepackage{enumerate}
\usepackage{tikz-cd} 
\usepackage{enumitem}
\usepackage{hyperref}
\usepackage{rotating}

\newtheorem{theorem}{Theorem}[section]
\newtheorem{lemma}[theorem]{Lemma}
\newtheorem{proposition}[theorem]{Proposition}
\newtheorem{corollary}[theorem]{Corollary}

\numberwithin{equation}{section}

\theoremstyle{definition}
\newtheorem{definition}[theorem]{Definition}

\newtheorem{remark}[theorem]{Remark}

\theoremstyle{remark}

\newcommand{\A}{\mathcal{A}}
\newcommand{\C}{\mathcal{C}}
\newcommand{\D}{\mathcal{D}}

\newcommand{\B}{\mathcal{B}}

\newcommand{\E}{\mathcal{E}}

\newcommand{\Z}{\mathcal{Z}}

\newcommand{\KK}{\mathfrak{K}}

\newcommand{\ZZ}{\mathfrak{Z}}
\newcommand{\CC}{\mathfrak{C}}

\newcommand{\vphi}{\varphi}

\newcommand{\Ad}{\operatorname{Ad}}
\newcommand{\id}{\operatorname{id}}

\newcommand{\Tr}{\operatorname{Tr}}
\newcommand{\diag}{\operatorname{diag}}
\newcommand{\Img}{\operatorname{Im}}

\begin{document}

\title{Strongly self-absorbing $C^*$-algebras and Fra\"\i ss\'e limits}

\author{Saeed Ghasemi}
\address{Institute of Mathematics, Czech Academy of Sciences, Czech Republic}
\email{\texttt{ghasemi@math.cas.cz}}
\thanks{This research is supported by the GA\v CR project 19-05271Y and RVO: 67985840}

\maketitle

\begin{abstract}
    We show that the Fra\"\i ss\'e limit of a category of unital separable $C^*$-algebras which is sufficiently closed under tensor products of its objects and morphisms is strongly self-absorbing, given that it has approximately inner  half-flip. We use this connection between  Fra\"\i ss\'e limits and strongly self-absorbing $C^*$-algebras to give a rather elementary proof for the well known fact that the Jiang-Su algebra is strongly self-absorbing.
    
    \ 


\end{abstract}

\section{Introduction}
A separable unital $C^*$-algebra $\D$ is called ``strongly self-absorbing" if it is not $*$-isomorphic to $\mathbb C$ and there is a $*$-isomorphism $\vphi: \D \to \D \otimes \D$ which is approximately unitarily equivalent to  $\id_\D \otimes 1_\D$. The UHF-algebras of infinite type, the Cuntz algebras $\mathcal O_2$, $\mathcal O_\infty$ and the Jiang-Su algebra $\Z$ are all strongly self-absorbing.
These $C^*$-algebras play a central role in Elliott's classification program of separable nuclear $C^*$-algebras by the K-theoretic data. In fact, the classification program has been almost exclusively focused on the class of (separable, unital and nuclear) $C^*$-algebras $\A$ that tensorially absorb a strongly self-absorbing $C^*$-algebra $\D$ (such $\A$ is called $\D$-stable). 
Strongly self-absorbing $C^*$-algebras are systematically studied in \cite{Toms-Winter}. They are automatically simple, nuclear and are either purely infinite or stably finite with at most one tracial state. Among strongly self-absorbing $C^*$-algebras the Jiang-Su algebra $\Z$ 
has received special attention in recent years. This is mostly due to the fact that strongly self-absorbing $C^*$-algebras are all $\Z$-stable \cite{Winter-2011} and therefore the class of $\Z$-stable $C^*$-algebras is the largest possible class of $\D$-stable $C^*$-algebras, for any strongly self-absorbing $C^*$-algebra $\D$. The remarkable classification of separable, simple, unital, nuclear, and $\Z$-stable $C^*$-algebras satisfying the UCT by K-theoretic invariants, is the pinnacle of the classification program (cf. \cite{Winter-classification}). 

The Jiang-Su algebra is a simple, separable, unital, nuclear, projectionless and infinite-dimensional $C^*$-algebra which has a unique tracial state (monotracial) and it is KK-equivalent to the $C^*$-algebra of complex numbers.
In their original paper \cite{Jiang-Su}, Jiang and Su define $\Z$ as the unique inductive limit of a sequence of dimension-drop algebras and unital $*$-homomorphisms which is simple and monotracial. This definition involves first constructing such a sequence and then showing that its limit has to be unique up to isomorphism via a classification result (\cite[Theorem 6.2]{Jiang-Su}). Since then many different characterizations of $\Z$ have been given (cf. \cite{Z-revisited}, \cite{Dadarlat-Toms} and \cite{Jacelon-Winter}).

It was shown already in \cite{Jiang-Su} that $\Z$ is strongly self-absorbing.
The original proof of Jiang and Su consists of essentially two main and mostly independent steps. One step is to show that $\Z$ has approximately inner  half-flip; recall that a $C^*$-algebra $\D$ has approximately inner  half-flip if the first factor embedding $\id_\D \otimes 1_{\D} : \D \to \D \otimes \D$ is approximately unitarily equivalent to the second factor embedding $ 1_{\D} \otimes \id_\D : \D \to \D \otimes \D$. This part of the proof follows from the basic properties of the constructed sequence of prime dimension-drop algebras and the unital $*$-embeddings whose limit is $\Z$, using an intricate continuous functional calculus argument (this is proved in \cite[Proposition 8.3]{Jiang-Su}, see also Remark \ref{Z-aihf} for a rough sketch of the proof). The second step is to show that there is an asymptotically central sequence of inner automorphisms of $\Z$ (\cite[Corollary 8.6]{Jiang-Su}). This step is essentially carried out by showing that  every unital endomorphism of $\Z$ is approximately inner and that there exists a unital $*$-homomorphism from $\Z \otimes \Z$ into $\Z$. The proofs of both of these statements involve classification results and the use of advanced tools from KK-theory. Thus, it would be desirable to find a proof of the fact that $\Z$ is strongly self-absorbing, which replaces the second step by arguments that do not depend (at least not so heavily) on the K-theoretic data and the classification tools. Towards this goal, recently a new and easier proof has been discovered by A. Schemaitat \cite{Schemaitat}, which uses a characterization of $\Z$ as a stationary inductive limit
of a generalized prime dimension-drop algebras and a trace-collapsing endomorphism, given in \cite{Z-revisited}.

In this paper we give a different proof of the fact that $\Z$ is strongly self-absorbing, which follows a more general approach, and does not use any classification tools nor any characterization of $\Z$. For this, we will show that the original sequences that are constructed in \cite[Proposition 2.5]{Jiang-Su} by Jiang and Su to define $\Z$, can be chosen (note that the construction is not canonical) so that they are all ``Fra\"\i ss\'e sequences" of a category of the building blocks and, as a consequence, the uniqueness of their limit, $\Z$, follows from a standard approximate intertwining arguments. This is an adaptation of a critical observation made in \cite{Eagle} and  \cite{Masumoto} that $\Z$ can be realized as the ``Fra\"\i ss\'e limit" of the category of prime dimension-drop algebras and trace-persevering unital $*$-embeddings. The main theorem of this paper (Theorem \ref{thm-F-ssa-intro}) shows that, roughly speaking, the Fra\"\i ss\'e limits of categories of $C^*$-algebras that are ``sufficiently" closed under tensor products and have approximately inner  half-flip are strongly self-absorbing. Therefore an application of this theorem provides an alternative proof of the fact that $\Z$ is strongly self-absorbing, in which the second step of the original proof is replaced by straightforward intertwining arguments concerning the sequences of prime dimension-drop algebras and their tensor products. Let us also point out that our proof does not require any Fra\"\i ss\'e theory beyond the approximate back-and-forth arguments (the standard approximate intertwining arguments), since in this particular case the Fra\"\i ss\'e sequences are ``built by hand" and there is no need to invoke the main existence theorem of Fra\"\i ss\'e limits (see Theorem \ref{uni-hom}).

It is shown in \cite{Eagle} and  \cite{Masumoto} that the category $\ZZ$ whose objects are all prime dimension-drop algebras with fixed faithful traces $(\Z_{p,q}, \sigma)$ and the morphisms are unital trace-preserving $*$-embeddings, is a ``Fra\"\i ss\'e category" (see Definition \ref{Fraisse-cat-def}) and its Fra\"\i ss\'e limit is the Jiang-Su algebra $(\Z, \nu)$ with its unique trace $\nu$. This is equivalent to saying that there is a sequence of objects $A_n=(\Z_{p_n, q_n}, \tau_n)$ and morphisms $ \vphi_n^{n+1}: A_n \to A_{n+1}$ in $\ZZ$ such that $\Z$ is the inductive limit of the sequence $(A_n, \vphi_n^{n+1})$ and satisfies the following properties:
\begin{itemize}
    \item[$(*)$] for every $(\Z_{p,q}, \sigma)$ in $\ZZ$ there exists a unital trace-preserving $*$-embedding $\psi:(\Z_{p,q}, \sigma) \to A_n$, for some $n$,
    \item[($**$)] for every $n$ and unital trace-preserving $*$-embedding $\gamma: A_n \to (\Z_{p,q}, \sigma)$ and for every $\epsilon>0$ and a finite subset $F$ of $A_n$, there is a unital trace-preserving $*$-embedding $\eta: (\Z_{p,q}, \sigma) \to A_m$ for some $m>n$ such that 
$\|\eta \circ \gamma(a) - \vphi_{n}^m(a)\|< \epsilon$
for every $a$ in $F$, where $\vphi_n^m = \vphi_{m-1}^m \circ \dots \circ \vphi_n^{n+1}$.
 \begin{equation*} \label{diag-Z}
\begin{tikzcd}[column sep=small]
A_{1} \arrow[r,  "\vphi_{1}^{2}"] 
&A_2\arrow[r,  "\vphi_{2}^{3}"] 
&\dots \arrow[r] 
&A_n \arrow[dr,  "\gamma"] \arrow[rr, "\vphi_n^m"] 
&& A_m \arrow[r] & \dots  \\
&&&& (\Z_{p,q}, \sigma) \arrow[ur, dashed, "\eta"] && 
\end{tikzcd}
\end{equation*}
\end{itemize}

The existence of such a sequence $(A_n,\vphi_n^m)$ in $\ZZ$ follows from the a fundamental theorem in Fra\"\i ss\'e theory which, roughly speaking, states that any category which has the ``joint embedding property", the ``near amalgamation property" and satisfies a ``separability assumption" (see Section \ref{pre} for the details) contains a sequence exhibiting the general behavior of $(A_n,\vphi_n^m)$, i.e. satisfying the conditions  $(*)$ and $(**)$ in the respective category. In a category $\KK$, whenever such sequences exist, a standard approximate intertwining argument guarantees that they all have isomorphic limits. Such sequences in $\KK$ are usually called ``Fra\"\i ss\'e sequences of $\KK$", revealing the origin of their mainstream study, and what usually makes them interesting is the uniqueness, universality and the (almost) homogeneity of their shared limit---the so called  ``Fra\"\i ss\'e limit of $\KK$".

In the category $\ZZ$, the Fra\"\i ss\'e limit is automatically the Jiang-Su algebra, along with its unique trace. This is because it is not difficult to see that the Fra\"\i ss\'e limit of $\ZZ$ (i.e. the inductive limit of any sequence in $\ZZ$ satisfying $(*)$ and $(**)$) is a simple and monotracial $C^*$-algebra (cf. \cite{Masumoto}) and the characterization of the Jiang-Su algebra as the unique unital simple and monotracial inductive limit of sequences of prime dimension-drop algebras (Theorem 6.2 of \cite{Jiang-Su}) implies that the limit has to be the Jiang-Su algebra. 
However, one does not need Fra\"\i ss\'e theory nor the mentioned characterization of $\Z$ to prove the existence of a sequence satisfying $(*)$ and $(**)$, or the fact that the limit of such a sequence is $\Z$. Instead, consider the originally constructed sequences intended to define $\Z$ in \cite[Proposition 2.5]{Jiang-Su}. The proposition states that certain sequences of prime dimension-drop algebras exist whose limits are simple and monotracial. Let us assume for a moment that we do not know whether all of these sequences have the same limit.
We will show that at the recursive stages of constructing such a sequence one can easily make sure that, with appropriately chosen traces, the constructed sequence would satisfy $(*)$ and $(**)$ as well as the conditions of \cite[Proposition 2.5]{Jiang-Su}. These sequences are described in Proposition \ref{my-Z-seq}. 
As pointed out earlier, then a standard approximate intertwining argument shows that these sequences have a unique limit up to $*$-isomorphism. This limit is the ``Jiang-Su algebra" since such sequences satisfy the conditions of the original construction.
 
The connection between strongly self-absorbing $C^*$-algebras and Fra\"\i ss\'e limits was suspected in \cite{Eagle} (see Problem 7.1). In pursuit of such connection,  in Section \ref{ssa-fl-sec} we show that if $\CC$ is a category of unital separable $C^*$-algebras and unital $*$-homomorphisms, which happens to have a Fra\"\i ss\'e limit $\D$ with approximately inner  half-flip, then $\D$ is strongly self-absorbing, if $\CC$ is, in a sense, cofinal (\emph{dominating}) in a category $\KK$ which contains $\CC$ (both the objects and morphisms) and all the objects of the form $\A \otimes \A$,  for $\A$ an object of $\CC$, and a ``modest but sufficient" set of new morphisms; we call such $\KK$ a $\otimes$-expansion of $\CC$ (Definition \ref{tensor expansion}). 
 \begin{theorem}\label{thm-F-ssa-intro}
Suppose $\CC$ is a category of unital separable $C^*$-algebras and unital $*$-homomorphisms and $\CC$  dominates a $\otimes$-expansion of itself. If $\mathfrak C$ has a Fra\"\i ss\'e limit $\D$ which has approximately inner  half-flip, then $\D$ is strongly self-absorbing.
\end{theorem}
 The notion of a category dominating a larger category (Definition \ref{dom-def-1}) was introduced and used in the context of Fra\"\i ss\'e categories by W. Kubi\'s (cf. \cite{Kubis-FS}). When a category $\CC$ dominates a larger category $\KK$, then any Fra\"\i ss\'e sequence of $\CC$ (whenever it exists) is also a Fra\"\i ss\'e sequence of $\KK$. This is very helpful, since often it is easier to show that the smaller category $\CC$ has a Fra\"\i ss\'e sequence and dominates $\KK$ than showing directly that $\KK$ has a Fra\"\i ss\'e sequence. In the proof of Theorem \ref{thm-F-ssa-intro} we use a weaker version of the notion of Fra\"\i ss\'e sequences, called the ``weak Fra\"\i ss\'e sequences" (Definition \ref{def-wF-seq}). Weak Fra\"\i ss\'e sequences are studied in \cite{Kubis-weak} and they also have isomorphic limits, called the ``weak Fra\"\i ss\'e limit". An outline of the proof of Theorem \ref{thm-F-ssa-intro} goes as follows. Suppose a category $\CC$ of unital separable $C^*$-algebras and unital $*$-homomorphisms dominates a category $\KK$ which is a $\otimes$-expansion of $\CC$, and $\D$ is the limit of a Fra\"\i ss\'e sequence $(\D_n, \vphi_n^m)$ in $\CC$. Since $\KK$ is dominated by $\CC$, we know that $(\D_n, \vphi_n^m)$ is also a Fra\"\i ss\'e sequence of $\KK$. The fact that  $\KK$ is a $\otimes$-expansion of $\CC$ guarantees that the sequence $(\D_n\otimes \D_n, \vphi_n^m \otimes \vphi_n^m)$ is a $\KK$-sequence, whose limit is clearly $\D\otimes \D$. Then we will show that $(\D_n\otimes \D_n, \vphi_n^m \otimes \vphi_n^m)$ is a weak Fra\"\i ss\'e sequence of $\KK$. Since $(\D_n, \vphi_n^m)$ is also a  weak Fra\"\i ss\'e sequence of $\KK$, by the uniqueness of the weak Fra\"\i ss\'e limit, $\D$ is self-absorbing. Then we use an easy version of the homogeneity property of the weak Fra\"\i ss\'e limit to show that $\D$ is in fact strongly self-absorbing (the author does not know whether $(\D_n\otimes \D_n, \vphi_n^m \otimes \vphi_n^m)$ is a Fra\"\i ss\'e sequence of $\KK$, however it is a ``weak" one). 

The last two sections are devoted to give a proof of the fact that $\Z$ is strongly self-absorbing, using Theorem \ref{thm-F-ssa-intro}. 
In Section \ref{sec-Z is fl} we directly show that $\Z$ is the Fra\"\i ss\'e limit of the category $\ZZ$. Finally, in the last section we define a $\otimes$-expansion $\mathfrak T$ of $\ZZ$ which is dominated by $\ZZ$.  Then Theorem \ref{thm-F-ssa-intro} implies that $\Z$ is strongly self-absorbing, since $\Z$ has approximately inner  half-flip. 

\null

{\bf Acknowledgment.} I would like to thank Ilijas Farah, Marzieh Forough and Alessandro Vignati for the helpful conversations and the anonymous referee for the suggestions and remarks.

\section{preliminaries}\label{pre}
Suppose $\KK$ is a category.
 We refer to the objects and morphisms of $\KK$ by $\KK$-objects and $\KK$-morphisms, respectively, and sometimes write $\A\in \KK$ if $\A$ is a $\KK$-object. A $\KK$-sequence $(\A_n, \vphi_n^{n+1})$ is a sequence of $\KK$-objects  $\A_n$ and $\KK$-morphisms $\vphi_n^{n+1}: \A_n \to \A_{n+1}$, for every $n$. We often denote such a sequence by $(\A_n, \vphi_n^m)$, where $\vphi_n^m: \A_n \to \A_m$ is defined by $\vphi_n^m= \vphi_{m-1}^{m} \circ \dots \circ \vphi_n^{n+1}$ for every $m\geq n$ (let $\vphi_n^n = \id_{\A_n}$). 
 By \emph{limit} we mean the inductive limit (also called \emph{colimit}). In categories of $C^*$-algebras and $*$-homomorphisms or more generally categories of Banach spaces and norm-decreasing (1-Lipschitz) linear maps, the limits of  sequences always exist (in a possibly larger category). If $\A$ is the limit of the $\KK$-sequence $(\A_n, \vphi_n^m)$,  let $\vphi_n^{\infty}: \A_n \to \A$ denote the induced inductive limit morphism (which may not be $\KK$-morphisms).

{\bf Notations.} Suppose $\A$ and $\B$ are normed structures. For $i=0,1$ and morphisms $\vphi_i: \A \to \B$ and $F\subseteq \A$ we sometimes write $\vphi_0 \approx_{\epsilon, F} \vphi_1$ if and only if of $\|\vphi_0(a) - \vphi_1(a)\|<\epsilon$, for every $a\in F$. We denote the image of the set $F$ under a morphism $\vphi$ by $\vphi[F]$. We also write $F\Subset \A$ if $F$ is a finite subset of  $\A$.
\subsection{Fra\"\i ss\'e sequences and the uniqueness}
For the rest of this section $\CC$ and $\KK$ are always categories of separable $C^*$-algebras and $*$-embeddings (injective $*$-homomorphisms).  Note that the following notions can be defined more generally for any metric category with 1-Lipschitz morphisms, cf. \cite{Kubis-ME}. Let us start with the main definition. 
 \begin{definition} \label{Fraisse-seq-def}
 A $\KK$-sequence $(\D_n, \vphi_n^m)$ is called a \emph{Fra\"\i ss\'e sequence} of $\KK$ if it satisfies,
\begin{itemize}
    \item[($\mathscr U$)] for every $\A\in \KK$ there exists a $\KK$-morphism $\vphi: \A \to \D_n$, for some $n\in \mathbb N$,
    \item[($\mathscr A$)] for every $\epsilon>0$, $n\in \mathbb N$,  $F\Subset \D_n$ and for every $\KK$-morphism $\gamma: \D_n \to \B$ there are $m\geq n$ and a $\KK$-morphism $\eta: \B \to \D_m$ such that
    $ \vphi_n^m \approx_{\epsilon, F} \eta \circ \gamma$.
\end{itemize}
 \end{definition}

A standard approximate intertwining argument (more generally known as ``approximate back-and-forth", in model theory) shows that two Fra\"\i ss\'e sequences must have $*$-isomorphic limits. 

 \begin{theorem} \label{F-unique}
The  Fra\"\i ss\'e sequences of $\KK$ (whenever they exist)  have  $*$-isomorphic limits. Moreover, if $(\D_n, \vphi_n^m)$ and $(\E_n, \psi_n^m)$ are both Fra\"\i ss\'e sequences of $\KK$ with limits $\D$ and $\E$, respectively, and $\theta: \D_k \to \E_\ell$ is a $\KK$-morphism, then for every $\epsilon>0$ and $F\Subset \D_k$ there is a $*$-isomorphism $\Phi: \D \to \E$ such that the diagram
 \begin{equation*}  
\begin{tikzcd}
\D_k \arrow[r,  "\vphi_k^\infty"] \arrow[d,  "\theta"]  &   \D \arrow[d, "\Phi"] \\ 
\E_\ell \arrow[r, "\psi_\ell^\infty"]  & \E 
\end{tikzcd}
\end{equation*}
$\epsilon$-commutes on $F$, i.e. $ \psi_k^\infty \circ \theta \approx_{\epsilon,F} \Phi \circ \vphi_\ell^\infty$.
\end{theorem}
\begin{proof}
Suppose $(\D_n, \vphi_n^m)$ and $(\E_n, \psi_n^m)$ are both Fra\"\i ss\'e sequences of $\KK$ with respective limits $\D$ and $\E$.
 Recursively construct sequences $(m_i)_{i\in\mathbb N}$, $(n_i)_{i\in\mathbb N}$ of natural numbers, finite sets
$F'_i\Subset  \D_{n_i}$, $G'_i\Subset \E_{m_i}$ and $\KK$-morphisms $\gamma_i: \D_{n_i} \to \E_{m_i}$ and $\eta_i: \E_{m_i} \to \D_{n_{i+1}}$ such that for every $i$,
\begin{itemize}
    \item $\D= \overline{\bigcup_i \vphi_{n_i}^{\infty}[F'_i]}$ and $\E= \overline{\bigcup_i \psi_{m_i}^{\infty}[G'_i]}$, 
    \item $\vphi_{n_i}^{n_{i+1}}[F'_i]\subseteq F'_{i+1}$ and $\psi_{m_i}^{m_{i+1}}[G'_i]\subseteq G'_{i+1}$,
    \item $\gamma_i[F'_i] \subseteq G'_{i}$ and $\eta_i[G'_i] \subseteq F'_{i+1}$,
    \item $\eta_i \circ \gamma_i \approx_{\frac{1}{2^{i-1}}, F'_i} \vphi_{n_i}^{n_{i+1}}$ and $\gamma_{i+1} \circ \eta_i \approx_{\frac{1}{2^i}, G'_i} \psi_{m_i}^{m_{i+1}}$.
\end{itemize}
This guarantees the existence of a $*$-isomorphism $\Phi: \D \to \E$ (cf. \cite[Proposition 2.3.2]{Rordam-Stormer}). 
 \begin{equation*}  
\begin{tikzcd}
\D_{n_1} \arrow[rr,  "\vphi_{n_1}^{n_2}"] \arrow[dr,  "\gamma_1"] & & \D_{n_2} \arrow[rr,  "\vphi_{n_2}^{n_3}"] \arrow[dr,  "\gamma_2"] && \D_{n_3} \arrow[r] & \dots & \D  \arrow[d, "\Phi"] \\ 
& \E_{m_1} \arrow[rr, "\psi_{m_1}^{m_2}"] \arrow[ur, "\eta_1"]  & & \E_{m_2} \arrow[rr, "\psi_{m_2}^{m_3}"] \arrow[ur, "\eta_2"] & & \dots & \E
\end{tikzcd}
\end{equation*}
To start, let $\epsilon_n=2^{-n}$, fix $F_n\Subset \D_n$ and $G_n\Subset \E_n$ such that 
$$\D= \overline{\bigcup_n \vphi_n^{\infty}[F_n]} \qquad \text{and} \qquad \E= \overline{\bigcup_n \psi_n^{\infty}[G_n]}.$$ 
Let $n_1 = 1$ and $F'_1 = F_1$. Using the condition ($\mathscr U$) for the Fra\"\i ss\'e sequence $(\E_n, \psi_n^m)$, there are $m_1$ and  $\gamma_1: \D_{n_1} \to \E_{m_1}$ in $\KK$. By the condition ($\mathscr A$) for the Fra\"\i ss\'e sequence $(\D_n, \vphi_n^m)$, find $\eta_1: \E_{m_1} \to \D_{n_2}$ in $\KK$, for some $n_2>n_1$,  such that
$$
\eta_1 \circ \gamma_1 \approx_{\epsilon_1, F'_{1}} \vphi_{n_1}^{n_2}.
$$
Let $G'_1=\gamma_1[F'_1] \cup G_{m_1}.$
Similarly, using the condition ($\mathscr A$) for $(\E_n, \psi_n^m)$, find $m_2>m_1$ and  $\gamma_2: \D_{n_2} \to \E_{m_2}$ such that 
$$
\gamma_2 \circ \eta_1 \approx_{\epsilon_2, G'_1} \psi_{m_1}^{m_2}.
$$
Let $F'_2= \eta_1[G'_1]\cup F_{n_2}\cup \vphi_{n_1}^{n_2}[F'_1]$. Again,  find $\eta_2:  \E_{m_2} \to \D_{n_3}$, for some $n_3>n_2$ such that 
$$
\eta_2 \circ \gamma_2 \approx_{\epsilon_3, F'_2} \vphi_{n_2}^{n_3}.
$$
Let  $G'_2=  \gamma_2[F'_2]\cup G_{m_2}\cup \psi_{m_1}^{m_2}[G'_1]$.
Continuing this process gives us the required approximate intertwining. For the second statement, in the proof above let $n_1= k$, $m_1=\ell$, $F_1 = F$, $\gamma_1=\theta$ and pick $\epsilon_i$ so that $\sum_{i=1}^\infty \epsilon_i <\epsilon$.
\end{proof}
The unique limit of Fra\"\i ss\'e sequences of $\KK$ is called the \emph{Fra\"\i ss\'e limit} of $\KK$.
The following notion was introduced in \cite{Kubis-FS} for abstract categories.
\begin{definition}\label{dom-def-1}
A category $\mathfrak C$ \emph{dominates}  $\KK$ if $\mathfrak C$ is a subcategory of $\KK$ and for any $\epsilon >0$ the following conditions hold:
 \begin{itemize}
     \item[($\mathscr C$)]  for every $\A\in \KK$ there are $\C\in \mathfrak C$ and $\vphi: \A \to \C$ in $\KK$,  i.e. $\mathfrak C$ is cofinal in $\KK$,
     \item[($\mathscr D$)] for every $\vphi: \A \to \B$  in $\KK$  with $\A\in \mathfrak C$ and for every $F\Subset \A$, there exist  $\beta: \B \to \C$  in $\KK$ with $\C \in \mathfrak C$ and a $\mathfrak C$-morphism $\alpha: \A \to \C$ such that 
     $\alpha \approx_{\epsilon, F} \beta \circ \vphi.$
      \begin{equation*}  
\begin{tikzcd}
\A \arrow[r,  "\vphi"] \arrow[dr, dashed, "\alpha"]  & \B \arrow[d, dashed, "\beta"] \\ 
& \C 
\end{tikzcd}
\end{equation*}
 \end{itemize}
\end{definition}

\subsection{The existence of Fra\"\i ss\'e sequences} \label{F-cat-sec} In this section we define the notion of Fra\"\i ss\'e categories (only) for categories of separable $C^*$-algebras and $*$-embeddings. These categories are guaranteed to contain Fra\"\i ss\'e sequences. 
\begin{definition}\label{Fraisse-cat-def}
Suppose $\CC$ is a category. We say that
\begin{itemize}[leftmargin=*]
\item$\CC$ has \emph{the joint embedding property} (also sometimes called ``directed", which is a more appropriate terminology in this setting) if for $\A, \B \in \CC$ there are $\C\in \CC$ and $\CC$-morphisms $\vphi:\A \to \C$ and $\psi:\B \to \C$. 
\item $\CC$ has \emph{the near amalgamation property} if for every $\epsilon>0$, $\A\in \CC$, $F\Subset \A$ and $\CC$-morphisms $\vphi : \A \to \B$, $\psi: \A \to \C$ there are $\D\in \CC$ and $\CC$-morphisms $\vphi': \B \to \D$ and $\psi' : \C \to \D$ such that $\|\vphi'\circ \vphi(a) , \psi' \circ \psi(a)\|<\epsilon$ for every $a\in F$
 \item $\CC$ is \emph{separable} if $\CC$ is dominated by a countable subcategory (a subcategory with countably many objects and morphisms).
\end{itemize}
A category is called a \emph{Fra\"\i ss\'e category} if it has the joint embedding property, the near amalgamation property and it is separable. 
\end{definition} 
It is well known that Fra\"\i ss\'e categories have Fra\"\i ss\'e sequences and their Fra\"\i ss\'e limits are unique, universal and almost homogeneous in the respective categories. To see a proof of the next theorem, see \cite{Kubis-FS}. Although in the statements and definitions of \cite{Kubis-FS} there are no finite sets $F$ and $\epsilon$ is always $0$, since the objects of our categories are separable, a routine adjustment of each proof immediately implies the corresponding statement.
In fact, the proofs of universality and almost homogeneity follow an approximate intertwining argument.

\begin{theorem}\label{uni-hom}
Assume $\CC$ is a Fra\"\i ss\'e category. Then $\CC$ has a Fra\"\i ss\'e sequence. If $(\D_n, \vphi_n^m)$ is a Fra\"\i ss\'e sequence of $\CC$ and $\D= \varinjlim (\D_n, \vphi_n^m)$ is the unique Fra\"\i ss\'e limit of $\CC$, then 
\begin{itemize}[leftmargin=*]
\item $\D$ is universal: if  $\B$ is the limit of a $\CC$-sequence, then there is a  $*$-embedding $\vphi: \B \to \D$, and
\item $\D$ is almost homogeneous: for every $\epsilon>0$, $\CC$-objects $\A,\B\subseteq \D$,  every $F\Subset \A$ and $*$-isomorphism $\vphi: \A \to \B$ in $\CC$, there is a $*$-automorphism $\theta: \D \to \D$ such that $\theta  \approx_{\epsilon, F} \vphi$. 
\end{itemize}

\end{theorem}

\begin{proposition}\label{Fraisse-dominate}
Suppose a category $\mathfrak C$   dominates $\KK$ and  $(\D_n, \vphi_n^m)$ is a  Fra\"\i ss\'e sequence of $\mathfrak C$, then 
 $(\D_n, \vphi_n^m)$ is also a Fra\"\i ss\'e sequence of $\KK$.
\end{proposition}

\begin{proof}
 The fact that $(\D_n, \vphi_n^m)$ satisfies the condition ($\mathscr U$) in $\KK$ follows from ($\mathscr C$) and the fact that  $(\D_n, \vphi_n^m)$ satisfies ($\mathscr U$) in $\CC$. For the condition ($\mathscr A$), suppose $\epsilon>0$, $F\Subset \D_n$ and a $\KK$-morphism $\gamma:\D_n \to \B$ are given. By ($\mathscr D$) find $\C\in \CC$, a $\mathfrak K$-morphism $\beta: \B \to \C$ and a $\CC$-morphism $\alpha: \D_n \to \C$ such that  $\beta \circ \gamma\approx_{\epsilon/2,F} \alpha$. Since $(\D_n, \vphi_n^m)$ is a Fra\"\i ss\'e sequence  of $\mathfrak C$, there is a $\CC$-morphism $\eta':\C \to \D_m$ in $\CC$, for some $m> n$, such that
$\vphi_n^m  \approx_{\epsilon/2,F} \eta' \circ \alpha$. 
\begin{equation*}
\begin{tikzcd}
\D_n \arrow[rr, "\vphi_n^m"] \arrow[dr, "\gamma"] \arrow[ddr,  "\alpha"]  & &  \D_m  \\ 
& \B \arrow[d, "\beta"] \\
& \C \arrow[uur, dashed, ',  "\eta'"]
\end{tikzcd}
\end{equation*}
Then we have $\vphi_n^m  \approx_{\epsilon,F} \eta \circ \gamma$, where $\eta$ is the $\KK$-morphism defined by $\eta = \eta' \circ \beta$. Therefore $(\D_n, \vphi_n^m)$ a Fra\"\i ss\'e sequence of $\KK$. 
\end{proof}
\begin{remark}
 The Fra\"\i ss\'e theory was introduced by R. Fra\"\i ss\'e \cite{Fraisse} to study the correspondence between countable  first-order homogeneous structures and properties of the classes of their finitely generated substructures.
It was extended to metric structures by  Ben Yaacov \cite{Ben Yaacov} in continuous model theory, and by  Kubi\'s \cite{Kubis-ME} in the framework of (metric-enriched) categories. The category-theoretical approach was recently used in \cite{Ghasemi-Kubis} to study the AF-algebras that are Fra\"\i ss\'e limits of categories of finite-dimensional $C^*$-algebras.
\end{remark}

\subsection{Weak Fra\"\i ss\'e sequences and the uniqueness} \label{def-wF-seq}
A $\KK$-sequence $(\D_n, \vphi_n^m)$ is called a \emph{weak Fra\"\i ss\'e sequence} of $\KK$ if it satisfies,
\begin{itemize}
    \item[($\mathscr U$)] for every $\A\in \KK$ there exists a $\KK$-morphism $\vphi: \A \to \D_n$, for some $n$,
    \item[($\mathscr {WA}$)] for every $\epsilon>0$, $n\in \mathbb N$,  $F\Subset \D_n$ there is a natural number $m\geq n$ such that  for every $\gamma: \D_m \to \B$ in $\KK$, there are $k> m$ and $\eta: \B \to \D_k$ in $\KK$ such that
    $ \vphi_n^k \approx_{\epsilon, F} \eta \circ \gamma \circ \vphi_n^m$.
\end{itemize}
The limit of a weak Fra\"\i ss\'e sequence of $\KK$ is called \emph{weak Fra\"\i ss\'e limit} of $\KK$. The following is proved in \cite[Lemma 4.9]{Kubis-weak} for abstract categories.

 \begin{theorem} \label{w-F-unique}
   (uniqueness) A weak Fra\"\i ss\'e limit of $\KK$, whenever it exists, is unique up to $*$-isomorphism.
   
   Suppose $(\D_i, \vphi_i^j)$ and $(\E_i, \psi_i^j)$ are both weak Fra\"\i ss\'e sequences of $\KK$ with respective limits $\D$ and $\E$. For every $\epsilon>0$, $n\in \mathbb N$ and $F\Subset \D_n$ there is $k\geq n$ such that for every $\theta: \D_k \to \E_\ell$ in $\KK$ there is a $*$-isomorphism $\Phi: \D \to \E$ such that the diagram
 \begin{equation*}  
\begin{tikzcd}
\D_n \arrow[r, "\vphi_n^k"] &\D_k \arrow[r,  "\vphi_k^\infty"] \arrow[d,  "\theta"]  &   \D \arrow[d, "\Phi"] \\ 
 & \E_\ell \arrow[r, "\psi_\ell^\infty"]  & \E 
\end{tikzcd}
\end{equation*}
$\epsilon$-commutes on $F$, i.e. $\psi_\ell^\infty \circ \theta \circ \vphi_n^k \approx_{\epsilon,F} \Phi \circ  \vphi_n^\infty$.
\end{theorem}
\begin{proof}
 Suppose $(\D_i, \vphi_i^j)$ and $(\E_i, \psi_i^j)$ are weak Fra\"\i ss\'e sequences of $\KK$, with respective limits $\D$ and $\E$. First, to show the uniqueness, we recursively find sequences $(n_i)_{i\in\mathbb N}$, $(k_i)_{i\in\mathbb N}$, $(m_i)_{i\in\mathbb N}$, $(\ell_i)_{i\in\mathbb N}$ of natural numbers, finite sets
$F'_i\Subset  \D_{k_i}$, $G'_i\Subset \E_{\ell_i}$ and $\KK$-morphisms $\gamma_i: \D_{k_i} \to \E_{m_i}$ and $\eta_i: \E_{\ell_i} \to \D_{n_{i+1}}$ such that for every $i$ we have (see Diagram (\ref{diag-intertwining})),
\begin{itemize}
    \item $n_1= 1$, $n_i\leq k_i < n_{i+1}$ and $m_i \leq \ell_i < m_{i+1}$,
    \item $\D= \overline{\bigcup_i \vphi_{k_i}^{\infty}[F'_i]}$ and $\E= \overline{\bigcup_i \psi_{\ell_i}^{\infty}[G'_i]}$, 
    \item $\vphi_{k_i}^{k_{i+1}}[F'_i]\subseteq  F'_{i+1}$ and $\psi_{\ell_i}^{\ell_{i+1}}[G'_i]\subseteq G'_{i+1}$,
    \item $\psi_{m_i}^{\ell_i} \circ \gamma_i[F'_i] \subseteq G'_{i}$ and $\vphi_{n_{i+1}}^{k_{i+1}} \circ\eta_i[G'_i] \subseteq F'_{i+1}$,
    \item $\eta_i \circ \psi_{m_i}^{\ell_i} \circ \gamma_i \approx_{\frac{1}{2^i}, F'_i} \vphi_{k_i}^{n_{i+1}}$,
    \item $\gamma_{i+1} \circ \vphi_{n_{i+1}}^{k_{i+1}} \circ \eta_i \approx_{\frac{1}{2^i}, G'_i} \psi_{\ell_i}^{m_{i+1}}$.
\end{itemize}

 \begin{equation}  \label{diag-intertwining}
\begin{tikzcd}
\D_{1} \arrow[r, "\vphi_1^{k_1}"] &\D_{k_1} \arrow[r,"\vphi_{k_1}^{n_2}"] \arrow[d,"\gamma_1"]  
& \D_{n_2} \arrow[r, "\vphi_{n_2}^{k_2}"]  
& \D_{k_2}\arrow[r,"\vphi_{k_2}^{n_3}"] \arrow[d,  "\gamma_2"]
&\D_{n_3}\arrow[r, "\vphi_{n_3}^{k_3}"] 
& \D_{k_3} \arrow[d,"\gamma_3"] \arrow[r]
&\dots  \\ 
\E_1 \arrow[r, "\psi_1^{m_1}"]
& \E_{m_1} \arrow[r,"\psi_{m_1}^{\ell_1}"]  
& \E_{\ell_1} \arrow[r, "\psi_{\ell_1}^{m_2}"] \arrow[u, "\eta_1"] 
& \E_{m_2}\arrow[r, "\psi_{m_2}^{\ell_2}"] 
&\E_{\ell_2}\arrow[u, "\eta_2"] \arrow[r, "\psi_{\ell_2}^{m_3}"]
&\E_{m_3}\arrow[r]
&\dots
\end{tikzcd}
\end{equation}
Then the $\KK$-morphisms $\alpha_i: \D_{k_i} \to \E_{\ell_i}$ and $\beta_i: \E_{\ell_i} \to \D_{k_{i+1}}$ given by $\alpha_i = \psi_{m_i}^{\ell_i} \circ \gamma_i$ and $\beta_i= \vphi_{n_{i+1}}^{k_{i+1}} \circ\eta_i$ produce an approximate intertwining between the two sequences, which guarantees the existence of a $*$-isomorphism $\Phi: \D \to \E$  (cf. \cite[Proposition 2.3.2]{Rordam-Stormer}). 

To start, let $\epsilon_n=2^{-n}$ and fix sequences $F_n\Subset \D_n$ and $G_n\Subset \E_n$ such that 
$$
\vphi_n^{n+1}[F_n] \subseteq F_{n+1}, \qquad  \qquad \psi_n^{n+1}[G_n] \subseteq G_{n+1},
$$
and
$$\D= \overline{\bigcup_n \vphi_n^{\infty}[F_n]}, \qquad  \qquad \E= \overline{\bigcup_n \psi_n^{\infty}[G_n]}.$$ 
Let $n_1=1$. Since $(\D_i, \vphi_i^j)$ is a weak Fra\"\i ss\'e sequence of $\KK$, we can find $k_1\geq 1$ such that
\begin{itemize}
    \item[(1)]  the condition ($\mathscr {WA}$) holds for $\epsilon_1$, $n_1$ and $F_1$ at $k_1$: that is, for every $\gamma:\D_{k_1} \to \B$ in $\KK$, there are $k> k_1$ and $\eta: \B \to \D_k$ in $\KK$ such that $\eta \circ \gamma \circ \vphi_1^{k_1} \approx_{\epsilon_1, F_1} \vphi_1^{k}$. 
\end{itemize}
Let $F'_1 = \vphi_1^{k_1}[F_1]$.
Using the condition ($\mathscr U$) for the weak Fra\"\i ss\'e sequence $(\E_i, \psi_i^j)$, there is a natural number $m_1\geq 1$ and a $\KK$-morphism  $\gamma_1: \D_{k_1} \to \E_{m_1}$. 
Again, since $(\E_i, \psi_i^j)$ is a weak Fra\"\i ss\'e sequence, find $\ell_1\geq m_1$ such that for the sequence $(\E_i, \psi_i^j)$ 
\begin{itemize}
    \item[(2)] the condition ($\mathscr {WA}$) holds for $\epsilon_1$, $m_1$ and $G_{m_1} \cup\gamma_1[F'_1]$ at $\ell_1$.
\end{itemize}
Let $G'_1= \psi_{m_1}^{\ell_1} [G_{m_1} \cup\gamma_1[F'_1]]$. By (1) there are $n_2>k_1$ and a $\KK$-morphism $\eta_1: \E_{\ell_1} \to \D_{n_{2}}$ such that (notice the choice of $F'_1$)
$$
\eta_1 \circ \psi_{m_1}^{\ell_1} \circ \gamma_1 \approx_{\epsilon_1, F'_1} \vphi_{k_1}^{n_{2}}.
$$
Find $k_2\geq n_2$ such that for the sequence $(\D_i, \vphi_i^j)$
\begin{itemize}
    \item[(3)] the condition ($\mathscr {WA}$) holds for $\epsilon_2$, $n_2$ and $F_{n_2} \cup \eta_1[G'_1]$ at $k_2$.
\end{itemize}
Let $F'_2 = \vphi_{n_2}^{k_2}[F_{n_2} \cup \eta_1[G'_1]]$. By (2) there are $m_2>\ell_1$ and a $\KK$-morphism $\gamma_2: \D_{k_2} \to \E_{m_2}$ such that 
$$
\gamma_2 \circ \vphi_{n_2}^{k_2} \circ \eta_1 \approx_{\epsilon_1, G'_1} \psi_{\ell_1}^{m_2}.
$$
Find $\ell_2\leq m_2$ such that for the sequence $(\E_i,\psi_i^j)$
\begin{itemize}
    \item[(4)] the condition ($\mathscr {WA}$) holds for $\epsilon_2$, $m_2$ and $G_{m_2} \cup \psi_{\ell_1}^{m_2}[G'_1] \cup \gamma_2[F'_2]$ at $k_2$.
\end{itemize}
Let $G'_2= \psi_{m_2}^{\ell_2}[G_{m_2} \cup \gamma_2[F'_2]] \cup \psi_{\ell_1}^{\ell_2}[G'_1]$. By (3) there are $n_3 >k_2$ and a $\KK$-morphism $\eta_2: \E_{\ell_2}\to \D_{n_3}$ such that 
$$\eta_2 \circ \psi_{m_2}^{\ell_2} \circ \gamma_2 \approx_{\epsilon_2, F'_2} \vphi_{k_2}^{n_3}.$$
Continuing this process produces the required approximate intertwining according to Diagram (\ref{diag-intertwining}). 

For the second statement, in the proof above start with $n_1= n$, $F_1 = F$ and let $k=k_1$, $\gamma_1=\theta$ and pick $\epsilon_i$ so that $\sum_{i=1}^\infty \epsilon_i <\epsilon$.
\end{proof}

\subsection{The existence of weak Fra\"\i ss\'e sequences} Clearly in a category $\KK$ a Fra\"\i ss\'e sequence is also  weak Fra\"\i ss\'e  and therefore Fra\"\i ss\'e categories contain weak Fra\"\i ss\'e sequences. However, a weakening of the notions of the near amalgamation property and separability of a category is sufficient (and necessary) to guarantee the existence of weak Fra\"\i ss\'e sequences. These categories are called ``weak Fra\"\i ss\'e categories" and they are studied in \cite{Kubis-weak}. As with Fra\"\i ss\'e sequences, if a category $\CC$ dominates a larger category $\KK$, then a weak Fra\"\i ss\'e sequence of $\CC$ (if it exists) is also a weak Fra\"\i ss\'e sequence of $\KK$. This remains true even if $\CC$ ``weakly dominates" $\KK$ (Proposition \ref{w-dom-thm}).

\begin{definition}\label{w-dom-def}
A category $\mathfrak C$  \emph{ weakly dominates} $\KK$ if $\mathfrak C$ is a subcategory of $\KK$ and for any given $\epsilon >0$, we have
 \begin{itemize}
     \item[($\mathscr C$)]  for every $\A\in \KK$ there are $\C\in \mathfrak C$ and $\vphi: \A \to \C$ in $\KK$,  i.e. $\mathfrak C$ is cofinal in $\KK$,
     \item[($\mathscr {WD}$)] for every $\A\in \mathfrak C$ and for every $F\Subset \A$, there exist $\zeta: \A \to \A'$ in $\mathfrak C$ such that for every $\vphi:\A' \to \B$ in $\KK$, there are  $\beta: \B \to \C$  in $\KK$ with $\C \in \mathfrak C$ and $\alpha: \A \to \C$ in $\mathfrak C$ such that $\alpha \circ \zeta \approx_{\epsilon, F} \beta \circ \vphi \circ \zeta.$
 \end{itemize}
\end{definition}
The following notion of ``weak amalgamation property" has been first identified by Ivanov \cite{Ivanov}
and later independently by Kechris and Rosendal \cite{Kechris-Rosendal} in model theory.
A category $\CC$ is called weakly Fra\"\i ss\'e if 
\begin{itemize}[leftmargin=*]
    \item it has the joint embedding property,
  \item $\CC$ has the \emph{weak near amalgamation property}:  for every $\epsilon>0$, $\A\in \KK$ and $F\Subset \A$ there is $\zeta: \A \to \A'$ such that for any $\KK$-morphisms $\vphi : \A' \to \B$, $\psi: \A' \to \C$ there are $\D\in \KK$ and $\KK$-morphisms $\vphi': \B \to \D$ and $\psi' : \C \to \D$ such that $\|\vphi'\circ \vphi \circ \zeta(a) , \psi' \circ \psi \circ \zeta(a)\|<\epsilon$, for every $a\in F$.
  \item $\CC$ is \emph{weakly separable}: it is weakly dominated by a countable subcategory. 
\end{itemize}

The proof of the following theorem is a straightforward metric adaptation of Theorem 4.6 of \cite{Kubis-weak} (by adding $2^{-n}$ and finite subsets to the proof).
\begin{theorem}
 A category has a weak Fra\"\i ss\'e sequence if and only if it is weakly Fra\"\i ss\'e. 
\end{theorem}
 The abstract category theoretic version of the next proposition is \cite[Lemma 4.3]{Kubis-weak} and its proof can again be adjusted to work for the metric case; similar to the proof of Proposition \ref{Fraisse-dominate}.
\begin{proposition} \label{w-dom-thm}
Suppose a category $\mathfrak C$ weakly dominates $\KK$ and  $(\D_n, \vphi_n^m)$ is a weak Fra\"\i ss\'e sequence of $\mathfrak C$, then 
 $(\D_n, \vphi_n^m)$ is also a weak Fra\"\i ss\'e sequence of $\KK$.
\end{proposition}

\section{Strongly self-absorbing $C^*$-algebras and Fra\"\i ss\'e limits} \label{ssa-fl-sec}
For $C^*$-algebras $\A$ and $\B$ we let $\A\otimes \B$ to denote their ``minimal" (or ``spacial") tensor product.
A $C^*$-algebra  is called ``self-absorbing" if it is $*$-isomorphic to its (minimal) tensor product with itself. Recall that $*$-homomorphisms $\vphi_i: \A \to \B$ $(i=1,2)$ between separable $C^*$-algebras $\A,\B$ are \emph{approximately unitarily equivalent} if there is a sequence $(u_n)_{n\in\mathbb N}$ of unitaries in the multiplier algebra of $\B$ such that 
$$
\lim_{n\to\infty}\|u_n^* \vphi_1(a)u_n -\vphi_2(a)\| = 0,
$$
for every $a\in \A$.
\begin{definition} Let $\D$ be a separable unital $C^*$-algebra.
\begin{itemize}[leftmargin=*]
    \item   $\D$ has \emph{approximately inner  half-flip} if the maps $\id_\D \otimes 1_\D$ and $1_\D \otimes \id_\D$ from $\D$ to $\D\otimes \D$ are approximately unitarily equivalent.
    \item $\D$ is \emph{strongly self-absorbing} if $\D\not\cong \mathbb C$ and there is a $*$-isomorphism $\vphi: \D \to \D \otimes \D$ which is approximately unitary equivalent to $\id_\D \otimes 1_\D$.
\end{itemize}
\end{definition}

\begin{remark}
Any strongly self-absorbing $C^*$-algebra $\D$ has \emph{approximately inner  flip} \cite{Toms-Winter} which is a stronger notion than approximately inner  half-flip and states that the flip $*$-automorphism $\sigma_\D: \D \otimes \D \to \D \otimes \D$, which sends $a\otimes b$ to $b\otimes a$, is approximately unitarily equivalent to the identity map on $\D \otimes \D$. 
The notion of approximately inner  flip was studied for $C^*$-algebras by Effros and Rosenberg \cite{Effros-Rosenberg}, inspired by a profound result of Connes about the hyperfinite $II_1$ factor. Any $C^*$-algebra with approximately inner  (half-)flip is automatically  simple and nuclear (cf. \cite{Effros-Rosenberg} and \cite{Kirchberg-Phillips}). 
\end{remark}

We need the following elementary lemma.
\begin{lemma}\label{aihf-seq}
Suppose $(\D_i, \vphi_i^j)$ is a sequence of unital  $C^*$-algebras and unital $*$-embeddings and $\D=\varinjlim (\D_i, \vphi_i^j)$. Then $\D$ has 
approximately inner  half-flip if and only if for every $\epsilon>0$, $n\in \mathbb N$,  $F \Subset \D_n$ there is a unitary $u\in \D_m\otimes \D_m$ for some $m\geq n$ such that 
$$
\|\vphi_n^m(a) \otimes 1_{\D_m} - u^*(1_{\D_m} \otimes \vphi_n^m(a))u\|<\epsilon
$$
for every $a\in F$.
\end{lemma}
\begin{proof}
The inverse direction is trivial. For the forward direction, suppose $\epsilon>0$, $n\in \mathbb N$ and  $F \Subset \D_n$ are given. We can assume $\epsilon<1$ and $F$ is contained in the unit ball of $\D_n$. Find a unitary $v\in \D\otimes \D$ such that 
$$
\|\vphi_n^\infty(a) \otimes 1_\D - v^*(1_\D \otimes \vphi_n^\infty(a))v\|<\epsilon/5
$$
for every $a\in F$.
Note that $\D\otimes \D$ is the limit of the sequence $(\D_i\otimes \D_i, \vphi_i^j \otimes \vphi_i^j)$.
For some $m\geq n$ there is an element $b$ in $\D_m\otimes \D_m$ such that  $\|v - \vphi_m^\infty\otimes \vphi_m^\infty(b)\|< \epsilon/10$. Then we have $\|bb^* - 1_{\D_m\otimes \D_m}\|<\epsilon/5$ and $\|b^*b - 1_{\D_m\otimes \D_m}\|<\epsilon/5$. It is easy to check that there is a unitary $u\in \D_m\otimes \D_m$ such that $\|u-b\|<\epsilon/5$. For every $a\in F$ we have 
\begin{align*}
   \|\vphi_n^m(a) \otimes 1_{\D_m} &- u^*(1_{\D_m} \otimes \vphi_n^m(a))u\| \\
   &\leq \|\vphi_n^m(a) \otimes 1_{\D_m} - b^*(1_{\D_m} \otimes \vphi_n^m(a))b\|+2\epsilon/5 \\
   & = \|\vphi_n^\infty(a) \otimes 1_\D - \vphi_m^\infty\otimes \vphi_m^\infty(b)^*(1_\D \otimes \vphi_n^\infty(a))\vphi_m^\infty\otimes \vphi_m^\infty(b)\|+2\epsilon/5\\
   &\leq \|\vphi_n^\infty(a) \otimes 1_\D - v^*(1_\D \otimes \vphi_n^\infty(a))v\|+4\epsilon/5 <\epsilon.
\end{align*}

\end{proof}

In the following $\CC$ and $\KK$ are always categories of unital separable $C^*$-algebras and unital $*$-embeddings.

\begin{definition}\label{tensor expansion}
We say that $\KK$ is a \emph{$\otimes$-expansion} of $\CC$ if 
\begin{itemize}[leftmargin=*]
    \item $\CC$ is a subcategory of $\KK$,
    \item $\A \otimes \A\in \KK$ for every $\A\in \CC$,
    \item if $\vphi: \A \to \B$ is in $\CC$ then $\vphi \otimes \vphi : \A \otimes \A \to \B \otimes \B$ belongs to $\KK$,
    \item for every $\A\in\CC$ the first factor embedding $\id_\A \otimes 1_\A: \A \to \A \otimes \A$ and the second factor embedding $1_\A \otimes \id_\A: \A \to \A \otimes \A$ belong to $\KK$,
    \item for every $\B\in \KK$ and a unitary $u \in \B$, the inner $*$-automorphism $\Ad_u: \B \to \B$ belongs to $\KK$.
\end{itemize}
\end{definition}

Now we are ready to prove the main result of this section. The key tool in the proof is the uniqueness of the weak Fra\"\i ss\'e limit, whenever it exists in a category (Theorem \ref{w-F-unique}).

\begin{theorem}\label{thm-F-ssa}
Suppose $\CC$ is a category of unital separable $C^*$-algebras and unital $*$-embeddings and $\CC$ (weakly) dominates a $\otimes$-expansion of itself. If $\mathfrak C$ has a (weak) Fra\"\i ss\'e limit $\D$ with approximately inner  half-flip, then $\D$ is strongly self-absorbing.
\end{theorem}
\begin{proof}
Suppose $\D=\varinjlim_i (\D_i, \vphi_i^j)$, where $(\D_i, \vphi_i^j)$ is a (weak) Fra\"\i ss\'e sequence of $\CC$ and assume $\KK$ is a $\otimes$-expansion of $\CC$ which is (weakly) dominated by $\CC$. By Proposition \ref{Fraisse-dominate} (Proposition \ref{w-dom-thm}) the sequence $(\D_i, \vphi_i^j)$ is a (weak) Fra\"\i ss\'e sequence of the category $\KK$.
The sequence $(\D_i \otimes \D_i, \vphi_i^j \otimes \vphi_i^j)$ is a $\KK$-sequence and its limit is $\D \otimes \D$.  We will show that $(\D_i \otimes \D_i, \vphi_i^j \otimes \vphi_i^j)$ is also a weak Fra\"\i ss\'e sequence of $\KK$. Hence, the uniqueness of weak Fra\"\i ss\'e limits (Theorem \ref{w-F-unique})  implies that $\D$ is self-absorbing. Then we use the second statement of Theorem \ref{w-F-unique} in order to show that $\D$ is strongly self-absorbing.

To show that $(\D_i \otimes \D_i, \vphi_i^j \otimes \vphi_i^j)$ is a weak Fra\"\i ss\'e sequence of $\KK$, note that the condition ($\mathscr U$) of \ref{def-wF-seq} is satisfied; by the condition ($\mathscr C$), for any $\B\in \KK$ there exist $\A \in\CC$ and $\vphi: \B \to \A$ in $\KK$. Since $(\D_i, \vphi_i^j)$ satisfies ($\mathscr U$) in the category $\CC$, for some $m$ there is a $\KK$-morphism $\psi: \A \to \D_m$. Then the map $(\id_{\D_m} \otimes 1_{\D_m}) \circ \psi \circ \vphi$ is a $\KK$-morphism from $\B$ to $\D_m \otimes \D_m$. 

To see that $(\D_i \otimes \D_i, \vphi_i^j \otimes \vphi_i^j)$ satisfies ($\mathscr {WA}$), suppose $\epsilon>0$, $n\in \mathbb N$ and $F\Subset \D_n\otimes \D_n$ are given. Let $G$ be a finite subset of $\D_n$ such that $F\subseteq_{\epsilon/2} G\otimes G$, where
$$
 G\otimes G :=\{\sum_{i=1}^{k} a_i \otimes b_i : a_i, b_i \in G\}.
$$
Since $\D$ has approximately inner  half-flip, by Lemma \ref{aihf-seq} we can find $m\geq n $ and a unitary $u\in \D_m\otimes \D_m$ such that 
\begin{equation}\label{eq1}
  \|u (\vphi_n^m(a) \otimes 1)u^* - 1 \otimes \vphi_n^m(a)\|<\epsilon /2 
\end{equation}
for every $a\in G$. By increasing $m$, if necessary, we can assume that for $(\D_i, \vphi_i^j)$  the condition ($\mathscr {WA}$) holds in the category $\KK$ for $\epsilon/2$, $n$ and $G$ at $m$ (recall that $(\D_i, \vphi_i^j)$ is also the (weak) Fra\"\i ss\'e sequence of $\KK$), that is 
\begin{itemize}
    \item[(1)] for every $\alpha: \D_m \to \B$ in $\KK$ there exist $k>m$ and $\beta: \B \to \D_k$ in $\KK$ such that $\|\beta \circ \alpha \circ \vphi_n^m (a) -  \vphi_n^k (a)\|< \epsilon/2 $
    for every $a\in G$.
\end{itemize}
We claim that in the category $\KK$ the sequence $(\D_i \otimes \D_i, \vphi_i^j \otimes \vphi_i^j)$ satisfies the condition ($\mathscr {WA}$) for $\epsilon$, $n$ and $F$ at $m$. 
Without loss of generality and after possibly dividing $\epsilon$ by a constant, it is enough to show that 

\begin{itemize}
    \item[($*$)] for any $\KK$-morphism $\gamma: \D_m \otimes \D_m \to \B$ there are $k>m$ and $\eta : \B \to \D_k \otimes \D_k$ in $\KK$ such that $
\|\eta \circ \gamma \circ (\vphi_n^m \otimes \vphi_n^m)(x) -  \vphi_n^k \otimes \vphi_n^k(x)\|< \epsilon/2 
$
where $x$ is of the form  $a \otimes 1$ or $1\otimes a$  for $a \in G$.
\end{itemize}
Let $\iota_i^m: \D_m \to \D_m \otimes \D_m$ $(i=1,2)$ be the unital $*$-embeddings defined by 
$$
\iota_1^m(a) = a \otimes 1 \qquad \text{and} \qquad \iota_2^m (b) = 1 \otimes b.
$$
The map $\gamma \circ \iota_1^m : \D_m \to \B$ is a $\KK$-morphism, hence by (1) there are $k>m$ and a unital $*$-embedding $\eta' :\B \to \D_k$ in $\KK$  such that
\begin{equation}\label{eq2}
   \|\eta' \circ \gamma \circ \iota_1^m \circ\vphi_n^m (a) - \vphi_n^k (a)\|<\epsilon/2
\end{equation}
for every $a\in G$. Let $\lambda: \D_m \otimes \D_m \to \D_k$ be the $\KK$-morphism $\lambda =\eta' \circ \gamma$.

 \begin{equation*}  
\begin{tikzcd}
\D_{n} \arrow[rr, "\vphi_n^{m}"] 
&&\D_{m} \arrow[d,"\iota^{m}_{1}"]\arrow[rr,"\vphi_m^k"]  
&& \D_{k}  \arrow[d,  "\iota_1^k"]
\\ 
\D_n \otimes \D_n  \arrow[rr, "\vphi_n^m\otimes \vphi_n^m"]
&& \D_m \otimes \D_m  \arrow[rr, "\vphi_m^k\otimes \vphi_m^k"] \arrow[dr,', "\gamma"] 
&& \D_k \otimes \D_k \arrow[out=240,in=300,loop,swap,  "\Ad_{\widetilde u^*}"]\\
&&& \B \arrow[uur,', "\eta'" near start] 
\end{tikzcd}
\end{equation*}

Let $\theta: \D_m \otimes \D_m \to \D_k \otimes \D_k$ be the $*$-homomorphism defined by $\theta = (\lambda \circ \iota_2^m) \otimes \vphi_m^k$ ($\theta$ is not necessarily a $\KK$-morphism). Put $\widetilde u = \theta(u)$ and note that $\widetilde u$ commutes with the image of the map $\iota_1^k \circ \lambda \circ \iota_1^m$. By linearity, it is enough to check this for $ u = a \otimes b$: for every $c\in \D_m$ we have
\begin{align*}
    \widetilde u [\iota_1^k \circ \lambda \circ \iota_1^m (c)] 
    &= 
   [\lambda (1 \otimes a) \otimes \vphi_m^k(b)]  [\lambda (c\otimes 1) \otimes 1] \\
  & =   \lambda(c\otimes a) \otimes \vphi_m^k(b)= [\lambda (c\otimes 1) \otimes 1]  [\lambda (1 \otimes a) \otimes \vphi_m^k(b)] \\
   &= [\iota_1^k \circ \lambda \circ \iota_1^m (c)] \widetilde u.
\end{align*}
Applying $\theta$ to (\ref{eq1}) gives us
\begin{equation} \label{eq3}
    \|\widetilde u(\lambda(1 \otimes \vphi_n^m(a))\otimes 1)\widetilde u^* - 1 \otimes \vphi_n^k (a) \|<\epsilon/2,
\end{equation}
for every $a\in G$.

Take $x= a \otimes 1$ for $a\in G$. The inequality \ref{eq2} and the fact that $\widetilde u$ commutes with $\lambda(\vphi_n^m(a) \otimes 1) \otimes 1$ imply that
\begin{align*}
    \|\Ad_{\widetilde u^*}\circ\iota_1^k \circ \lambda \circ (\vphi_n^m \otimes \vphi_n^m)(x) -  ( \vphi_n^k  \otimes \vphi_n^k) (x)\| 
   & =\|\lambda(\vphi_n^m(a)\otimes 1)\otimes 1 - 
    \vphi_n^k(a)\otimes 1\| \\
    &= \|\lambda(\vphi_n^m(a)\otimes 1) - 
    \vphi_n^k(a)\| <\epsilon/2.
\end{align*}
Assume $x = 1 \otimes a$ for $a\in G$. Then by the inequality \ref{eq3}  we have
\begin{align*}
    \|\Ad_{\widetilde u^*}\circ\iota_1^k \circ \lambda \circ (\vphi_n^m \otimes \vphi_n^m)(x) - &( \vphi_n^k  \otimes \vphi_n^k) (x)\| \\
    &=\|\widetilde u(\lambda(1 \otimes \vphi_n^m(a))\otimes 1)\widetilde u^* -  1 \otimes \vphi_n^k (a)  \|<\epsilon/2.
\end{align*}
The $\KK$-morphism $\eta = \Ad_{\widetilde u^*} \circ \iota_1^m \circ \eta'$ satisfies $(*)$ as required and therefore by the uniqueness of the weak Fra\"\i ss\'e limit $\D\otimes \D$ is $*$-isomorphic to $\D$.

To see that $\D$ is strongly self-absorbing, for each natural number $i$ consider the natural induced $*$-embeddings $\vphi_i^\infty: \D_i \to \D$ and the $*$-embeddings $\theta_i: \D_i \to \D \otimes \D$ defined by $\theta_i=(\vphi_i^\infty \otimes \vphi_i^\infty) \circ \iota^i_2$. Fix finite subsets $G_i$ of $\D_i$ such that $\vphi_i^{i+1}[G_i] \subseteq G_{i+1}$ and $\bigcup_i^\infty \vphi_i^\infty[G_i]$ is dense in $\D$.  Since $(\D_i, \vphi_i^j)$ is a weak Fra\"\i ss\'e sequence of $\KK$, by the second statement of Theorem \ref{w-F-unique}, for every $n$ there is $m\geq n$ and a $*$-isomorphism $\Phi_n : \D \to \D \otimes \D$ such that $\| \theta_m \circ \vphi_n^m (a) - \Phi_n \circ \vphi_n^\infty(a)\|<2^{-n}$ for every $a\in G_n$.
\begin{equation*}
\begin{tikzcd}
\D_n \arrow[r,  "\vphi_n^m"] 
& \D_m \arrow[r,"\vphi_m^\infty"] \arrow[d,  "\iota^{m}_2"]  &   \D  \arrow[d, ', "\Phi_n"]   \\ 
& \D_m \otimes \D_m \arrow[r, "\vphi_m^\infty \otimes \vphi_m^\infty"] 
& \D\otimes \D
\end{tikzcd}
\end{equation*}
Clearly for every $i$ and $j>i$, 
the diagram
\begin{equation*}
\begin{tikzcd}
\D_i \arrow[r,  "\vphi_i^\infty"] \arrow[d,  "\iota^{i}_2"]  &   \D_{j} \arrow[d, "\iota^{j}_2"]  \\ 
\D_i \otimes \D_i \arrow[r, "\vphi_i^j \otimes \vphi_i^j"]  & \D_{j}\otimes \D_{j}
\end{tikzcd}
\end{equation*}
commutes. Hence for every $n$ and $k,\ell>n$ we have  $\Phi_{k}|_{\vphi_n^\infty[G_n]} \approx_{1/2^n} \Phi_{\ell}|_{\vphi_n^\infty[G_n]}$. Therefore  $\Phi=\lim_n \Phi_n$ is a well-defined $*$-isomorphism from $\D $ onto $\D \otimes \D$. 

We claim that $\Phi$ is approximately unitarily equivalent to $\id_\D \otimes 1_\D$. 
Suppose $\epsilon>0$ and $F\Subset \D$ are given. Without loss of generality let us assume that $F$ is a subset of $\vphi_n^\infty[G_n]$, for a large enough $n$. Find $m\geq n$ such that
\begin{align}\label{eq5}
 \Phi(\vphi_n^\infty(a)) \approx_{\epsilon/3} \Phi_m(\vphi_n^\infty(a)) &\approx_{\epsilon/3} \vphi_m^\infty \otimes \vphi_m^\infty (1\otimes \vphi_n^m(a)) \\
\nonumber &= 1 \otimes \vphi_n^\infty(a)   
\end{align}
for every $a\in G_n$.
By increasing $m$, if necessary, find a unitary $u$ in $\D_m \otimes \D_m$ such that  
$$
  \|\vphi_n^m(a)\otimes 1 - u^* (1 \otimes \vphi_n^m(a)) u\|< \epsilon/3.
$$
Let $\widetilde u = \vphi_m^\infty \otimes \vphi_m^\infty(u)$. Apply $\vphi_m^\infty \otimes \vphi_m^\infty$ to the above inequality to get
\begin{equation}\label{eq6}
     \|\vphi_n^\infty(a)\otimes 1 - \widetilde u ^* (1 \otimes \vphi_n^\infty(a)) \widetilde u\|< \epsilon  /3 
\end{equation}
for every $a\in G_n$. Hence, from (\ref{eq5}) and (\ref{eq6}) we have
$$
  \|\vphi_n^\infty(a)\otimes 1 - \widetilde u ^* \Phi(\vphi_n^\infty(a)) \widetilde u\|< \epsilon
$$
for every $a\in G_n$. Therefore $\Phi$ is approximately unitarily equivalent to $\id_\D \otimes 1_\D$.
\end{proof}

\section{The Jiang-Su algebra as a Fra\"\i ss\'e limit} \label{sec-Z is fl}

Let us start by recalling some definitions and fundamental facts about dimension-drop algebras and the Jiang-Su algebra from \cite{Jiang-Su}. For every positive integer $n$, by $M_n$ we denote the $C^*$-algebra of all complex $n\times n$-matrices, with the unit $1_n$. For positive integers $p, q$ the \emph{dimension-drop algebra} $\Z_{p,q}$ is defined by 
$$
\Z_{p,q} = \{f\in C([0,1], M_{pq}): f(0)\in M_p \otimes 1_q \text{ and } f(1)\in 1_p \otimes M_q\}.
$$
A dimension-drop algebra is called \emph{prime} if $(p,q) = 1$.
A unital $*$-embedding  $\vphi: \Z_{p,q} \to \Z_{p', q'}$ between dimension-drop algebras is called \emph{diagonalizable} (in \cite{Masumoto}) if there are continuous maps $\{\xi_i: [0,1] \to [0,1]: i= 1, \dots, k\}$ and a unitary  $u\in C([0,1], M_{p'q'})$ such that 
$$
\vphi(f) = \Ad_{u} \circ\begin{bmatrix}
f\circ \xi_1 &  & 0 \\
& \ddots & \\
0 &  & f \circ \xi_k
\end{bmatrix}
$$
for every $f\in \Z_{p,q}$. Let $\Lambda^\vphi = \{\xi_1, \xi_2, \dots, \xi_k\}$ and define $\Delta^{\vphi}: \Z_{p,q} \to C([0,1], M_{p'q'})$  by
$$
\Delta^{\vphi}(f) = \begin{bmatrix}
f\circ \xi_1 &  & 0 \\
& \ddots & \\
0 &  & f \circ \xi_k
\end{bmatrix}
$$
for every $f\in \Z_{p,q}$. 

\begin{proposition}\label{Z-facts} \cite{Jiang-Su}
 There is a sequence  $(\A_n , \vphi_n^m)$ of prime dimension-drop algebras and  diagonalizable $*$-embeddings such that for every $\xi\in \Lambda^{\vphi_n^{m}}$  the diameter of the image of  $\xi$ is not greater than $1/2^{m-n}$.
   The limit of any such sequence  is unital, simple and has a unique tracial state.
\end{proposition}

In \cite{Jiang-Su} Jiang and Su use tools from the classification theory to show that there is up to $*$-isomorphisms a unique simple and monotracial (i.e. has a unique tracial state) $C^*$-algebra which is the limit of a sequence of prime dimension-drop algebras and unital $*$-embeddings. They denote the limit of a (any) sequence as in Proposition \ref{Z-facts} by $\Z$, which is nowadays called the Jiang-Su algebra. 
A closer look at the proof of Proposition \ref{Z-facts} in \cite{Jiang-Su} shows that one can choose the sequences so that they satisfy some extra properties.
\begin{proposition} \label{my-Z-seq}
 There is a sequence  $(\Z_{p_n, q_n} , \vphi_n^m)$ of prime dimension-drop algebras and  diagonalizable $*$-embeddings such that
 \begin{itemize}
     \item[($\dagger$)] for every $m\geq n$, if $\Lambda^{\vphi_n^{m}} = \{\xi_1 \dots, \xi_k\}$ then we have $\xi_1(x) \leq \xi_2(x) \dots \leq \xi_k(x)$ for all $x\in[0,1]$ and the diameter of the image of each  $\xi_i$ is not greater than $1/2^{m-n}$, and
     \item[$(\ddagger)$]  $p_mq_m$ is a multiple of $m$, for every natural number $m$. 
 \end{itemize}
\end{proposition}
\begin{proof}
It is been already pointed out in Remark 2.6 of \cite{Jiang-Su} that $\{\xi_i\}$ can be arranged in the increasing order. In the $(m+1)$-th stage of the recursive construction of the proof in \cite[Proposition 2.5]{Jiang-Su}, when the sequence $(\Z_{p_i,q_i},\vphi_i)_{i\leq m}$ is already picked, we can choose $k_0$ and $k_1$ such that
$$
k_0>2q_m, \quad k_1>2p_m, \quad (k_0p_m, k_1q_m)=1
$$
and moreover make sure that $k_0k_1$ is a multiple of $m+1$ (distribute the factors of the prime factorization of $m+1$ appropriately among $k_0$ and $k_1$). Then  $p_{m+1} = k_0 p_m$ and $q_{m+1} = k_1 q_m$ satisfy the property that  $p_{m+1}q_{m+1}$ is a multiple of $m+1$. 
\end{proof}
In Theorem \ref{Z is unique} we will show that sequences as in Proposition \ref{my-Z-seq} have $*$-isomorphic limits.

\null

\noindent{\bf The Category $\ZZ$.} We define the Category $\ZZ$ as in \cite{Eagle} and \cite{Masumoto}. Let $\ZZ$ denote the category whose objects are all prime dimension-drop algebras $(\Z_{p,q}, \tau)$ with fixed faithful traces. The set of $\ZZ$-morphisms from $(\Z_{p,q}, \tau)$ to $(\Z_{p',q'}, \tau')$ is the set of all unital trace-preserving $*$-embeddings.
 \begin{remark}
The main result of \cite{Masumoto}  
 states that $\ZZ$ is a Fra\"\i ss\'e category and its Fra\"\i ss\'e limit is simple and monotracial. Then the characterization of $\Z$ in \cite[Theorem 6.2]{Jiang-Su} is used to show that the Fra\"\i ss\'e limit of $\ZZ$ is in fact $\Z$.  Note that the fixed traces are only for the purpose of the near amalgamation property. Namely, the unital $*$-embeddings with the same prime dimension-drop algebra in their domains can be nearly amalgamated if for some traces on their respective codomains they induce the same trace on the domain.  
   \end{remark}

It turns out that one can show that $\Z$ is the Fra\"\i ss\'e limit of $\ZZ$ directly, without using the mentioned characterization of $\ZZ$. Instead, in Theorem \ref{Z is unique} we show that any sequence as in Proposition \ref{my-Z-seq} can be turned into a Fra\"\i ss\'e sequence of $\ZZ$, by fixing appropriate traces.
The proof of Theorem \ref{Z is unique} is similar to the one of \cite[Proposition 4.10]{Masumoto} of the fact that  $\ZZ$ has the near amalgamation property. As a result, to carry out the proof of Theorem \ref{Z is unique}, we need the same tools that are developed in \cite{Masumoto} (listed below in Proposition \ref{ZZ-facts}), all of which have straightforward proofs. 

First recall that, essentially by Riesz representation theorem, traces (by ``trace" we mean a tracial state) on a dimension-drop algebra $\Z_{p,q}$ correspond to probability (Radon) measures on $[0,1]$. In fact, the trace space $T(\Z_{p,q})$ is affinely homeomorphic to the space of all probability measures on $[0,1]$.  Given a probability measure $\tau$ on $[0,1]$, we use the same letter $\tau$ to denote the ``corresponding trace" on $\Z_{p,q}$ which defined by 
$$
\tau(A) = \int_0^1 \Tr(f(x)) d\tau(x)
$$
for every $f\in \Z_{p,q}$, where $\Tr$ is the unique trace on $M_{pq}$. A measure $\tau$ on $[0,1]$ is faithful if and only if the corresponding trace is faithful. A measure is called \emph{diffuse} if any measurable set of non-zero measure can be partitioned into two measurable sets of non-zero measures. For the maps $\xi,\zeta: [0,1]\to [0,1]$ we write $\xi\leq \zeta$ if and only if $\xi(x)\leq \zeta(x)$ for any $x\in [0,1]$.

\begin{proposition} \label{ZZ-facts}
\begin{enumerate}
    
    \item \cite[Proposition 4.3]{Masumoto} For every $(\Z_{p,q}, \tau)$ in $\ZZ$ there is a $\ZZ$-morphism $\psi:(\Z_{p,q}, \tau) \to (\Z_{p,q}, \sigma) $, for  any trace $\sigma$ on $\Z_{p,q}$ which corresponds to a diffuse measure on $[0,1]$.
    \item \cite[Proposition 4.4]{Masumoto} Suppose $p$ and $q$ are coprime natural numbers. There exists $N\in \mathbb N$ such that if $p',q'$ are coprime natural numbers larger than $N$ and $pq$ divides $p'q'$, there exists a $\ZZ$-morphism $\vphi: (\Z_{p,q}, \tau) \to (\Z_{p',q'} ,\tau')$, for any faithful diffuse measures $\tau, \tau'$ on $[0,1]$.
    \item \cite[Proposition 4.7]{Masumoto} Suppose $\vphi: (\Z_{p,q}, \tau) \to  (\Z_{p',q'}, \tau')$ is a $\ZZ$-morphism and $pq$ divides $p'q'$. For every $\epsilon>0$ and $F\Subset \Z_{p,q}$  there is a diagonalizable $\ZZ$-morphism $\psi:  (\Z_{p,q}, \tau) \to  (\Z_{p',q'}, \tau')$ such that $\|\vphi(f) - \psi(f)\|<\epsilon$ for every $f\in F$ and $\Lambda^{\psi}= \{\xi_1, \dots, \xi_k\}$ satisfies $\xi_1\leq \dots \leq \xi_k$.
    \item \cite[Proposition 4.8]{Masumoto} Suppose $\tau, \tau'$ correspond to diffuse faithful measures on $[0,1]$. For every $\epsilon>0$ and $(\Z_{p,q}, \tau)\in \ZZ$ there is a diagonalizable $\ZZ$-morphism $\psi:  (\Z_{p,q}, \tau) \to  (\Z_{p',q'}, \tau')$ such that the diameter of the image of each $\xi\in \Lambda^{\psi}$ is less than $\epsilon$.    
    \item \cite[Lemma 4.2]{Masumoto} Suppose  $\Z_{p,q}$ and  $\Z_{p',q'}$ are prime dimension-drop algebras such that $pq$ divides $p'q'$. There is a unitary $w\in C([0,1], M_{p'q'})$ such that for any $\psi: \Z_{p,q} \to C([0,1], M_{p'q'})$ of the form 
    $$\psi(f)= \diag(f\circ\zeta_1, \dots, f\circ \zeta_k),$$
    where $\zeta_1\leq \dots \leq \zeta_k$ are continuous maps from $[0,1]$ into $[0,1]$, the image of $\Ad_{w} \circ\psi$ is included in $\Z_{p',q'}$.
    \item  \cite[Lemma 4.9]{Masumoto} Suppose $\vphi, \psi: \Z_{p,q} \to \Z_{p',q'}$ are diagonalziable $*$-embeddings such that $\Delta^{\vphi} = \Delta^\psi$. For any $\epsilon>0$ and $F\Subset \Z_{p,q}$ there exists a unitary $w$ in $C([0,1], M_{p',q'})$ such that the inner $*$-automorphism $\Ad_w$ preserves $\Z_{p',q'}$ and $\|\Ad_w \circ \psi(f) - \vphi(f)\|<\epsilon$ for every $f\in F$.
\end{enumerate}
\end{proposition}

\begin{remark}
 The composition $\vphi' \circ \vphi$ of two diagonalizable morphism is again diagonalizable. In fact, if $\Lambda^\vphi=\{\xi_1, \dots \xi_k\}$ and $\Lambda^{\vphi'}=\{\xi'_1, \dots \xi'_{k'}\}$, then we can arrange so that $\Lambda^{\vphi'\circ \vphi}=\{\xi_1\circ\xi'_1, \dots, \xi_1\circ\xi'_{k'},  \dots ,\xi_k\circ\xi'_1, \dots, \xi_k\circ\xi'_{k'}\}$. Hence if the diameter of the image of each map in $ \Lambda^{\vphi}$ is less than $\epsilon$, then the diameter of the image of each map in $\Lambda^{\vphi' \circ \vphi}$ is also less than $\epsilon$.
\end{remark}

 \begin{theorem} \label{Z is unique}
Sequences as in Proposition \ref{my-Z-seq} have $*$-isomorphic limits. 
\end{theorem}
\begin{proof}
Suppose $(\A_n, \vphi_n^m)$ is a sequence as in Proposition \ref{my-Z-seq} and $\A_n=\Z_{p_n,q_n}$. 
 Fix a sequence $\{\nu_n: \nu_n\in T(\A_n)\}_{n\in\mathbb N}$ of traces which correspond to diffuse measures on $[0,1]$. For every $n$ the sequence $\{ \nu_m \circ \vphi_n^m\}_{m\in \mathbb N}$ of traces on $\A_n$ has the property that $\{ \nu_m \circ \vphi_n^m(f)\}_{m\in \mathbb N}$ forms a norm-Cauchy sequence for every $f\in \A_n$ (cf. \cite[Proposition 2.8]{Jiang-Su}). 
 Therefore $\tau_n = \lim_{m\to \infty} \nu_m \circ \vphi_n^m$ is a well-defined trace on $\A_n$, which also corresponds to a diffuse measure on $[0,1]$. Each $\vphi_n^m:(\A_n , \tau_n) \to (\A_m, \tau_m)$ is a trace-preserving $*$-embedding, since $\tau_n = \tau_m \circ \vphi_n^m$, for all $m\geq n$. The limit of the sequence $(\A_n, \vphi_n^m)$ is a monotracial $C^*$-algebra and  $\tau=\lim_n \tau_n$ is its unique trace.

 Let $A_n=(\A_n, \tau_n)$. Then $(A_n, \vphi_n^m)$ is a $\ZZ$-sequence. We claim that $(A_n, \vphi_n^m)$ is a Fra\"\i ss\'e sequence of $\ZZ$. The condition ($\mathscr U$) of Definition \ref{Fraisse-seq-def} is clear from (1) and (2) of Proposition \ref{ZZ-facts}, since by $(\ddagger)$ for any pair $p,q$ of coprime natural numbers,  $p_nq_n$ divides $pq$, if $n>pq$. To see ($\mathscr A$), suppose $\epsilon>0$, $n$, $F\Subset \A_n$ and a $\ZZ$-morphism 
 $\gamma: A_n \to (\Z_{p,q}, \sigma)$ are given. Take $\delta>0$ such that 
 $$
 \|f(x) - f(y)\|< \epsilon \quad \text{if} \quad |x-y|<\delta
 $$
 for every $f\in F$ and $x,y\in[0,1]$. Find $m>n$ such that $2^{n-m}<\delta$, i.e. the diameter of the image of each $\xi\in \Lambda^{\vphi_n^m}$ is less than $\delta$.  Suppose
 $$\vphi_n^m(f) = \Ad_u \circ \diag(f\circ\xi_1 ,\dots, f\circ \xi_\ell).$$
 By (1), (3) and (4) of Proposition \ref{ZZ-facts}, without loss of generality, we can assume that $\sigma$ corresponds to a diffuse measure, $\gamma$ is a diagonalizable $\ZZ$-morphism such that if 
 $\Lambda^{\gamma}=\{\zeta_1 ,\dots, \zeta_{k}\}$ then it satisfies 
   $\zeta_1\leq \dots \leq \zeta_{k}$, and the diameter of the image each $\zeta_i$ is less than $\delta/3$.
 
By increasing $m$, if necessary, we can make sure that $p_mq_m$ is a multiple of $pq$ and find a diagonalizable $\ZZ$-morphism $\gamma': (\Z_{p,q}, \sigma) \to A_m$ such that  $\Lambda^{\gamma'}=\{\zeta'_1 ,\dots, \zeta'_{k'}\}$ satisfies $\zeta'_1\leq \dots \leq \zeta'_{k'}$ and $kk'=\ell$. Therefore
 $$|\Lambda^{\gamma'\circ\gamma}|=|\Lambda^{\gamma}||\Lambda^{\gamma'}| = kk' =\ell =|\Lambda^{\vphi_n^m}|.$$
 Then,
 if $\Lambda^{\gamma'\circ\gamma}=\{\zeta''_1 ,\dots, \zeta''_\ell\}$ we have $\zeta''_1\leq \dots \leq \zeta''_\ell$  and
the diameter of the image each $\zeta''_i$ is less than $\delta/3$.

 We claim that $\|\xi_i - \zeta''_i\|< \delta$ for every $i\leq \ell$. If not, then for some $j\leq \ell$ and some $t\in[0,1]$ we have $\xi_j(t) \geq \zeta''_j(t) +\delta$. Set $c=\min \xi_{j+1}$ (if $j= d$ let $c=1$) and $d= \max \zeta_j''$. Note that
 \begin{itemize}
     \item  the image of $\zeta''_i$ is included in $[0,d]$ for every $1\leq i \leq j$,
     \item if $\Img(\xi_i) \cap [0,c) \neq \emptyset$, then $i\leq j$,
     \item $c> d+\delta/3$.
 \end{itemize}
Since $\vphi_n^m$ and $\gamma' \circ \gamma$ are trace-preserving
 $$
 j= \sum_{i=1}^j \tau_m((\zeta_j'')^{-1}[0,d]) \leq \ell \tau_n([0,d]) < \ell \tau_n([0,c]) \leq  
 \sum_{i=1}^j \tau_m((\xi_j)^{-1}[0,1])  =j,
 $$
  which is a contradiction. The claim implies that $\Delta^{\vphi_n^m}\approx_{\epsilon, F} \Delta^{\gamma' \circ \gamma}$.
  
 Suppose
 $\gamma'\circ \gamma(f) = \Ad_v \circ \diag(f\circ\zeta''_1 ,\dots, f\circ \zeta''_\ell).$
  By (5) of Proposition \ref{ZZ-facts}, there is a unitary $w$ such that the images of the maps $\vphi =\Ad_{w u^*} \circ \vphi_n^m$ and $\psi=\Ad_{w v^*} \circ (\gamma'\circ \gamma)$ are included in $\Z_{p',q'}$.
  Clearly 
  $$
\vphi= \Ad_w\circ \Delta^{\vphi_n^m} \quad \text{ and } \quad \psi = \Ad_w \circ \Delta^{\gamma' \circ \gamma}.
$$
  Since $\Delta^{\vphi}= \Delta^{\vphi_n^m}$ and $\Delta^{\psi} = \Delta^{\gamma' \circ \gamma}$, by (6) of Proposition \ref{ZZ-facts}, there are unitaries $w_0$ and $w_1$ such that the inner $*$-automorphism $\Ad_{w_0}$ and $\Ad_{w_1}$ both preserve $\Z_{p_m,q_m}$ and 
  \begin{align*}
      \| \vphi(f) - \Ad_{w_0} \circ\vphi_n^m(f)\|<\epsilon \quad \text{and} \quad 
      \| \psi(f) - \Ad_{w_1} \circ\gamma' \circ \gamma (f)\|<\epsilon,
  \end{align*}
  for any $f\in F$. Finally for any $f\in F$ we have
  \begin{align*}
      \Ad_{w_0^* w_1} \circ \gamma' \circ \gamma(f) & \approx_{\epsilon} \Ad_{w_0^*} \circ \psi (f) \\
      &=\Ad_{w_0^* w} \circ \Delta^{\gamma'\circ \gamma}(f) \\
      &\approx_{\epsilon} \Ad_{w_0^* w} \circ \Delta^{\vphi_n^m}(f) \\
      &= \Ad_{w_0^*} \circ \vphi(f) \\
      &\approx_{\epsilon} \vphi_n^m(f).
  \end{align*}
  Therefore the $\ZZ$-morphism $\eta= \Ad_{w_0^* w_1}\circ \gamma'$ satisfies $\eta \circ \gamma\approx_{3\epsilon,F} \vphi_n^m$, as required. This finished the proof.
\end{proof}
In fact, we have shown the following.
\begin{corollary}\label{ZZ has F-seq}
The category $\ZZ$ contains Fra\"\i ss\'e sequences and the Fra\"\i ss\'e limit of $\ZZ$ is $(\Z, \nu)$, where $\nu$ is the unique trace of $\Z$.
\end{corollary}

\section{$\Z$ is strongly self-absorbing}
In this short section first we define a category $\mathfrak T$ which is a $\otimes$-expansion of $\ZZ$ and it is dominated by $\ZZ$. This would complete all the necessary ingredients to use Theorem \ref{thm-F-ssa} in order to prove that $\Z$ is strongly self-absorbing, as it has the approximately inner  half-flip.

\begin{remark}\label{Z-aihf}
In \cite[Proposition 8.3]{Jiang-Su}, Jiang and Su proved that $\Z$ has approximately inner  half-flip. Let us provide a sketch of their proof.  Suppose a sequence $(\Z_{p_n,q_n}, \vphi_n^m)$ as in  Proposition \ref{Z-facts} (whose limit is $\Z$) and  $\epsilon >0$ and a finite set $F\subseteq \Z_{p_n,q_n}$ are given.  We can find $m>n$ large enough so that
$$
\|\Delta^{\vphi_n^m}(f)(x) - \Delta^{\vphi_n^m}(f)(y) \|< \epsilon/3 \qquad \forall x,y\in [0,1] \text{ and } f\in F.
$$
If $\vphi_n^m = u^* \Delta^{\vphi_n^m}u$ then define a continuous path of unitaries  $U:[0,1]^2 \to M_{p_m^2q^2_m}$ by
$$
U(x,y) = (u^*(x)u(y) \otimes 1_{p_mq_m})S
$$
where $S$ is the unitary on $M_{p_mq_m}\otimes M_{p_mq_m}$ which implements the flip (that is, $S^*(a\otimes b)S = b\otimes a$ for every $a,b\in M_{p_mq_m}$). Then it is easy to check that
\begin{equation} \label{eq-aihf}
 \|U^*(\varphi_n^m(f) \otimes 1_{p_mq_m})U - 1_{p_mq_m} \otimes \varphi_n^m(f) \| \qquad \forall  f\in F.   
\end{equation}
The problem is that such a unitary path $U$ does not belong to $\Z_{p_m,q_m} \otimes \Z_{p_m,q_m}$. However, a delicate and ingenious ``continuous functional calculus" argument can be applied to pinch the endpoints of $U$ and turn it into a unitary element of $\Z_{p_m,q_m} \otimes \Z_{p_m,q_m}$, which also satisfies (\ref{eq-aihf}) (see \cite[Proposition 8.3]{Jiang-Su} for details). 
\end{remark}

\null

\noindent{\bf The Category $\mathfrak T$.} Let $\mathfrak T$ denote a category whose objects are 
$$
\ZZ \cup \{(\Z_{p,q} \otimes \Z_{p,q}, \tau \otimes \tau): (\Z_{p,q},\tau) \in \ZZ\}
$$
and $\mathfrak T$-morphisms are exactly finite compositions of the maps of the form below.
\begin{enumerate}[label=(\roman*)]
    \item  $\ZZ$-morphisms,
    \item  $\Ad_u: (\Z_{p,q} \otimes \Z_{p,q},\tau \otimes \tau) \to (\Z_{p,q} \otimes \Z_{p,q}, \tau\otimes \tau)$  for every unitary $u\in \Z_{p,q} \otimes \Z_{p,q}$ and for every $(\Z_{p,q}, \tau)\in \ZZ$, 
    \item  the $*$-embeddings $\id \otimes 1_{pq}: (\Z_{p,q},\tau) \to  (\Z_{p,q} \otimes \Z_{p,q},\tau \otimes \tau)$ and $1_{pq} \otimes \id:(\Z_{p,q},\tau )\to  (\Z_{p,q} \otimes \Z_{p,q},\tau \otimes \tau)$, for every $(\Z_{p,q}, \tau)\in \ZZ$.
    \item $\vphi \otimes \vphi:  (\Z_{p,q} \otimes \Z_{p,q},\tau \otimes \tau) \to (\Z_{p',q'} \otimes \Z_{p',q'}, \tau'\otimes \tau')$ for every unital trace-preserving $*$-embeddings (a $\ZZ$-morphism) $\vphi:(\Z_{p,q},\tau) \to (\Z_{p',q'},\tau')$,
    \item  the unital $*$-homomorphism $\rho: (\Z_{p,q} \otimes \Z_{p,q},\tau\otimes \tau) \to (\Z_{p^2,q^2},\tau)$ defined by  $\rho(g)(t) = g(t,t)$, for every $(\Z_{p,q}, \tau)\in \ZZ$.
\end{enumerate}

Note that $\Z_{p,q} \otimes \Z_{p,q}$ is not $*$-isomorphic to any dimension-drop algebra, since its center is $C([0,1]^2)$.
The following lemma justifies why $\mathfrak T$ is in fact a category which contains $\ZZ$ as a full subcategory.

\begin{lemma}\label{lemma-zz}
 Any  $\mathfrak T$-morphism $\vphi:(\Z_{p,q},\tau) \to (\Z_{p',q'}, \tau')$ is a trace-preserving unital $*$-embedding i.e. $\vphi$ is a $\ZZ$-morphism.
 \end{lemma}
 \begin{proof}
    Any such $\vphi$ is either already a $\ZZ$-morphism or it is of the form 
    $$
    \begin{tikzcd}[column sep=0.1]
    (\Z_{p,q},\tau) \arrow[r, "\vphi_1"] & (\Z_{p_1,q_1},\tau_1) \arrow[r,"\id\otimes 1", "\text{or } 1\otimes \id"'] & (\Z_{p_1,q_1} \otimes \Z_{p_1,q_1},\tau_1\otimes \tau_1)\arrow[r,  "\vphi_2\otimes \vphi_2"] \arrow[out=120,in=60,loop,"\Ad_u"] & \dots & \\
     \dots\arrow[r, "\vphi_2\otimes \vphi_2"]& (\Z_{p_2,q_2} \otimes \Z_{p_2,q_2},\tau_2\otimes \tau_2)\arrow[r, "\rho"] \arrow[out=240,in=300,loop,swap,  "\Ad_v"] &(\Z_{p_2^2,q_2^2},\tau_2) \arrow[r,"\vphi_3"] &(\Z_{p',q'},\tau')  &&
    \end{tikzcd}
    $$
    for some $\vphi_1, \vphi_2,\vphi_3$  in $\ZZ$, unitaries $u$, $v$ and a $*$-homomorphism $\rho$ from (v). All of these maps except $\rho$ are trace-preserving $*$-embeddings. Therefore it is enough to show that $\vphi: (\Z_{p,q},\tau) \to (\Z_{p^2,q^2}, \tau)$ is trace-preserving, when $\vphi= \rho \circ (\id \otimes 1_{pq})$ (or $\vphi= \rho \circ (1_{pq}\otimes \id)$). First note that $\vphi$ is a $*$-embedding such that 
$$\vphi(f)(x) = f(x) \otimes 1_{pq}$$
for every $x\in [0,1]$ and $f\in \Z_{p,q}$. We have
\begin{align*}
\tau(\vphi(f)) & = \int_0^1 \Tr(f(x) \otimes 1_{pq}) d\tau(x)\\
&=\int_0^1 \Tr(f(x)) d\tau(x) = \tau(f).   
\end{align*}
\end{proof}

 Similarly, any $\mathfrak T$-morphism from  objects of the form  $(\Z_{p,q},\tau)$ to the objects of the form $(\Z_{p',q'} \otimes \Z_{p', q'} , \tau'\otimes \tau')$ is automatically a trace-preserving unital $*$-embedding.

\begin{lemma}\label{ZZ-expansion}
$\mathfrak T$ is a $\otimes$-expansion of $\ZZ$.
\end{lemma}
\begin{proof}
This is clear from Definition \ref{tensor expansion} and the definition of $\mathfrak T$.
\end{proof}

\begin{lemma}\label{Z>ZxZ}
The category $\ZZ$ is dominating in $\mathfrak T$.
\end{lemma}
\begin{proof}
The condition ($\mathscr C$) of Definition \ref{dom-def-1} is clear, since for any $(\Z_{p,q} \otimes \Z_{p,q},\tau \otimes \tau)\in\mathfrak T$ the diagonal map $\rho: (\Z_{p,q} \otimes \Z_{p,q},\tau \otimes \tau) \to (\Z_{p^2,q^2},\tau)$ of the form (v) is a $\mathfrak T$-morphism.

For ($\mathscr D$), note that by Lemma \ref{lemma-zz}, any  $\mathfrak T$-morphism $\vphi:(\Z_{p,q},\tau) \to (\Z_{p',q'}, \tau')$ is in $\ZZ$ and if $\psi: (\Z_{p,q},\tau) \to (\Z_{p',q'}\otimes \Z_{p',q'}, \tau'\otimes \tau')$ is a $\mathfrak T$-morphism, then $\rho \circ \psi: (\Z_{p,q},\tau) \to (\Z_{p'^2,q'^2}, \tau')$ is in $\ZZ$, where $\rho:(\Z_{p',q'}\otimes \Z_{p',q'}, \tau'\otimes \tau') \to (\Z_{p'^2,q'^2}, \tau')$ is again the diagonal map of the form (v). 
\end{proof}

\begin{corollary}
The Jiang-Su algebra $\Z$ is strongly self-absorbing.
\end{corollary}
\begin{proof}
The Jiang-Su algebra with its unique trace $(\Z, \nu)$ is the Fra\"\i ss\'e limit of $\ZZ$ (Corollary \ref{ZZ has F-seq}) and has approximately inner  half-flip (\cite[Proposition 8.3]{Jiang-Su}).  The category $\mathfrak T$ is a $\otimes$-expansion of $\ZZ$ (Lemma \ref{ZZ-expansion}) and $\ZZ$ dominates $\mathfrak T$ (Lemma \ref{Z>ZxZ}). 
Therefore it follows from  Theorem \ref{thm-F-ssa} that $\Z$ is strongly self-absorbing. 
\end{proof}

\bibliographystyle{amsplain}

\end{document}